\documentclass[preprint,12pt]{elsarticle}

\usepackage{amssymb}
\usepackage{amsmath,mathrsfs,amsfonts}
\usepackage{mathtools}
\usepackage{amsthm}
\usepackage{xcolor}

\usepackage{cases}
\usepackage{multirow}
\usepackage{comment}
\usepackage[colorlinks=true]{hyperref}
\usepackage{soul}
\usepackage{enumitem}  
\usepackage{subcaption}

\newtheorem{proposition}{Proposition}
\newtheorem{theorem}{Theorem}

\newtheorem{lemma}{Lemma}

\usepackage{hyperref}
\usepackage{tabularx}

\usepackage{xcolor}
\usepackage{tikz}
\usetikzlibrary{arrows.meta,patterns}
\tikzstyle{vertex}=[circle, draw, inner sep=0pt, minimum size=1.5em]
\tikzstyle{svertex}=[square, draw, inner sep=0pt, minimum size=1.5em]
\newcommand{\vertex}{\node[vertex]}

\usetikzlibrary{decorations}

\newcommand{\suchthat}{\;\ifnum\currentgrouptype=16 \middle\fi|\;}

\newcommand{\dist}{\operatorname{dist}}

\makeatletter
\newcommand{\vast}{\bBigg@{4}}
\newcommand{\Vast}{\bBigg@{2.5}}
\makeatother

\graphicspath{{images/}}

\journal{Operations Research Letters}

\begin{document}

\begin{frontmatter}

\title{Optimal design of vaccination policies: A case study for Newfoundland and Labrador}

\author[1]{Faraz Khoshbakhtian}

\author[2]{Hamidreza Validi}

\author[3]{Mario Ventresca}

\author[1,4]{Dionne Aleman}

\affiliation[1]{organization={Department of Mechanical \& Industrial Engineering, University of Toronto},
state={ON},
country={CA}}
            
\affiliation[2]{organization={Department of Industrial, Manufacturing \& Systems Engineering, Texas Tech University},
state={TX},
country={U.S.}}  
            
\affiliation[3]{organization={School of Industrial Engineering, Purdue University},
state={IN},
country={U.S.}}  

\affiliation[4]{organization={Institute for Pandemics, University of Toronto},
state={ON},
country={CA}}

\begin{abstract}
This paper proposes pandemic mitigation vaccination policies for Newfoundland and Labrador (NL) based on two compact mixed integer programming (MIP) models of the distance-based critical node detection problem (DCNDP). 
Our main focus is on two variants of the DCNDP that seek to minimize the number of connections with lengths of at most one (1-DCNDP) and two (2-DCNDP).
A polyhedral study for the 1-DCNDP is conducted, and new aggregated inequalities are provided for the 2-DCNDP. 
The computational experiments show that the 2-DCNDP with aggregated inequalities outperforms the one with disaggregated inequalities for graphs with a density of at least $0.5 \%$.
We also study the strategic vaccine allocation problem as a real-world application of the DCNDP and conduct a set of computational experiments on a simulated contact network of NL.
Our computational results demonstrate that the DCNDP-based strategies can have a better performance in comparison with the real-world strategies implemented during COVID-19.
\end{abstract}

\end{frontmatter}

\section{Introduction}
\label{intro}
Vaccine programs are the cornerstones of public health and play an indispensable role in preventing, stopping, or controlling the spread of infectious diseases~\cite{hussein2015vaccines}. 
The global impact of the recent COVID-19 pandemic underscores the urgency of developing vaccination policies to minimize human cost and the socioeconomic disruption of future infectious disease outbreaks~\cite{miller2020disease,ihme2021modeling}. 
Vaccines are among the most effective tools in pandemic mitigation but are usually in limited supply. 
The tension between the supply and the demand makes strategic vaccine rollout planning an essential part of any pandemic playbook. 
Improved vaccination policies can save lives, drastically reduce the number of infectious diseases, and significantly reduce the financial costs of pandemics~\cite{mulberry2021vaccine,bubar2021model}. 

There are multiple metrics to assess the effectiveness of a vaccination policy in a community.
The basic reproduction number ($R_0$) is a well-known metric extensively utilized by public health professionals for analyzing the structural complexities of community networks and evaluating the potential for disease propagation~\cite{ventresca2013evaluation}.
While an $R_0<1$ denotes that the number of infectious cases is decreasing, an $R_0>1$ indicates that the number of infectious cases is increasing. 
\citet{hurford2021} assumed an $R_0$ of 2.4 as the pre-pandemic basic reproduction number for the province of Newfoundland and Labrador.
In a network context, $R_0$ is a degree-based metric that is only reliant on immediate neighboring nodes. However, disease spread may be more accurately represented by considering the closeness of any two nodes, rather than whether or not they are directly linked, making distance-based metrics potentially more appropriate for investigation of vaccination policy effectiveness in mitigating disease spread~\cite{alozie2021efficient}.
To minimize the number of short connections, we employ the distance-based critical node detection problem (DCNDP) to decrease the spread of infectious diseases~\cite{salemi2022}, where critical nodes are vaccinated, thus preventing disease spread through them.

We conduct a polyhedral study and develop mixed integer programming (MIP) formulations for two variants of DCNDP to identify individuals for vaccination. 
DCNDP is a specific variant of the critical node detection problem (CNDP). 
CNDP is a well-known class of combinatorial optimization problems that identifies a subset of \textit{critical nodes} from a graph such that their removal optimizes a pre-defined measure of connectivity (e.g., pairwise connectivity among the nodes) over the residual graph~\cite{lalou2018critical}. 
This class of problems has natural applications in a wide range of real-world problems such as computational biology and drug discovery~\cite{boginski2009identifying}, communication networks~\cite{commander2007wireless}, and most relevant for the current work, to epidemic control~\cite{ventresca2015efficiently, charkhgard2018integer, tao2006epidemic}.

Letting $k$ as the hop parameter of the DCNDP, we denote $k$-DCNDP as the problem of minimizing the number of connections of length at most $k$.
We specifically consider $k \in \{1,2\}$, which provides unique opportunities to conduct polyhedral studies and propose new valid inequalities.
However, all previous studies of DCNDP considered $k \ge 3$. 
\citet{veremyev2015critical} introduced DCNDP and proposed four MIP formulations for the problem. 
~\citet{salemi2022} developed two MIP formulations for the problem and provided theoretical comparisons between their formulations and a MIP model proposed by~\citet{veremyev2015critical}.
~\citet{hooshmand2020efficient} provided a Benders decomposition framework to improve run times for instances with large diameters and weighted edges.
~\citet{alozie2021efficient} developed a path-based MIP formulation and efficient separation algorithms to solve the problem, and later proposed a heuristic approach~\cite{alozie2022heuristic}.
~\citet{aringhieri2019polynomial} provided polytime and pseudo-polytime algorithms for solving the problem on paths, trees, and series-parallel graphs.

Figure~\ref{DCNDP_example} illustrates optimal solutions of the 1-DCNDP and 2-DCNDP for two instances of the DCNDP on graph $G=(V,E)$ with budget (i.e., the maximum number of critical nodes that can be removed) $b=1$, vertex deletion cost $a_v = 1$ for every vertex $v \in V$, and connection cost $c_{uv} = 1$ for every vertex pair $\{u,v\} \in \binom{V}{2}$ with a connection of length at most $k$ for $k$-DCNDP. 

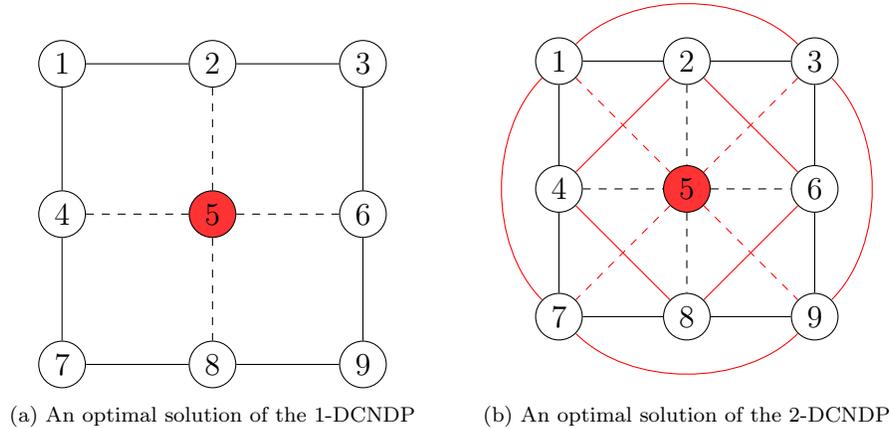
\begin{figure}[tbp]
\begin{center}

\begin{subfigure}{.45\textwidth}
\centering
\begin{tikzpicture}[scale=0.2]
\vertex(1) at (-5,24) [] {$1$};
\vertex(2) at (5,24) [] {$2$};
\vertex(3) at (15,24) [] {$3$};
\vertex(4) at (-5,14) [] {$4$};
\vertex[fill=red!80](5) at (5,14) [] {$5$};
\vertex(6) at (15,14) [] {$6$};
\vertex(7) at (-5,4) [] {$7$};
\vertex(8) at (5,4) [] {$8$};
\vertex(9) at (15,4) [] {$9$};

\draw (1) -- (2);
\draw (2) -- (3);

\draw (1) -- (4);
\draw [dashed] (2) -- (5);
\draw (3) -- (6);

\draw [dashed] (4) -- (5);
\draw [dashed] (5) -- (6);

\draw (4) -- (7);
\draw [dashed] (5) -- (8);
\draw (6) -- (9);

\draw (7) -- (8);
\draw (8) -- (9);
\end{tikzpicture}
\caption{An optimal solution of the 1-DCNDP}
\label{subfig:a}
\end{subfigure}
\begin{subfigure}{.45\textwidth}
\centering
\begin{tikzpicture}[scale=0.17]
\vertex(1) at (-5,24) [] {$1$};
\vertex(2) at (5,24) [] {$2$};
\vertex(3) at (15,24) [] {$3$};
\vertex(4) at (-5,14) [] {$4$};
\vertex[fill=red!80](5) at (5,14) [] {$5$};
\vertex(6) at (15,14) [] {$6$};
\vertex(7) at (-5,4) [] {$7$};
\vertex(8) at (5,4) [] {$8$};
\vertex(9) at (15,4) [] {$9$};

\draw (1) -- (2);
\draw (2) -- (3);

\draw [red] (2) -- (4);
\draw [red] (2) -- (6);
\draw [red] (4) -- (8);
\draw [red] (6) -- (8);

\draw [red, dashed] (1) -- (5);
\draw [red, dashed] (3) -- (5);
\draw [red, dashed] (5) -- (7);
\draw [red, dashed] (5) -- (9);

\draw (1) -- (4);
\draw [dashed] (2) -- (5);
\draw (3) -- (6);

\draw [dashed] (4) -- (5);
\draw [dashed] (5) -- (6);

\draw (4) -- (7);
\draw [dashed] (5) -- (8);
\draw (6) -- (9);

\draw [red] (1) to [out=45, in=135, distance = 6cm] (3);
\draw [red] (7) to [out=-45, in=225, distance = 6cm] (9);
\draw [red] (1) to [out=225, in=135, distance = 6cm] (7);
\draw [red] (3) to [out=-45, in=45, distance = 6cm] (9);

\draw (7) -- (8);
\draw (8) -- (9);
\end{tikzpicture}
\caption{An optimal solution of the 2-DCNDP}
\label{subfig:b}
\end{subfigure}
\end{center}
\caption{Connections of lengths one and two are shown in black and red, respectively. The red vertex and dashed lines show the removed vertex and connections, respectively. 
(a) An optimal solution of the 1-DCNDP in which removing vertex 5 decreases the number of 1-hop connections from 12 to 8. (b) An optimal solution of the 2-DCNDP in which removing vertex 5 decreases the number of connections with length at most two from 24 to 16.
}
\label{DCNDP_example}
\end{figure}

A removed vertex can be interpreted as a vaccinated person. 
As the physical social distancing was considered one or two meters during the COVID-19 pandemic in Canada~\citep{globalnews2021}, minimizing the number of low-hop connections is a reasonable objective in the vaccination context.
As COVID-19 is a very contagious disease, its spread rate is very high when one is within a close distance of an infected individual, thus motivating the study of the low-hop variants of the DCNDP (e.g., 1-DCNDP and 2-DCNDP).

We conduct a polyhedral study for the 1-DCNDP and propose aggregated valid inequalities for the 2-DCNDP to solve high-density instances of the problem. 
Our experiments show that the aggregated inequalities outperform their disaggregated variant when the density of the input graph is at least $0.5 \%$. 
To our knowledge, this is the first work that studies the 1-DCNDP and 2-DCNDP from the lens of mixed integer programming.

The DCNDP solutions, while valuable for identifying critical nodes in complex networks, are limited when it comes to formulating practical vaccination programs and policies. 
One fundamental limitation is that these solutions consider only the topology of contact networks, but in reality, contact network topology is not known exactly. Further, vaccination priority policies need to be easily understood and communicated, while a DCNDP is a set of highly specific individual nodes which may not correspond to actual members of a population, or it may difficult to communicate to only particular selected individuals that they should get vaccinated. 
To address this issue, our networks are generated using a granular agent-based simulation (ABS) model of NL~\cite{aleman2020morpop}, which overlays individual characteristics and demographics on each node. We then use decision trees to post-process DCNDP solutions and to provide policymakers with managerial insights and simple, implementable rules describing vaccine policies.

This combined approach of optimization and machine learning on realistic network sizes can enhance the effectiveness of vaccination programs and ensure that limited resources are allocated efficiently, ultimately contributing to successfully pandemic mitigation. However, we emphasize that our DCNDP analysis considers \emph{only} reduction of disease spread, and does not consider other important public health factors such as disease outcome severity or vaccination equity.

\section{MIP Formulations of 1-DCNDP and 2-DCNDP}\label{section: mipModels}
Given graph $G = (V, E)$ with vertex set size of $n:=|V|$, edge set size of $m:=|E|$, edge weights $w_e \ge 0$ for every $e \in E$, deletion budget $b \ge 0$, vertex deletion cost $a_v \ge 0$ for every vertex $v \in V$, connection cost $c_{uv} \ge 0$ for every vertex pair $\{u,v\} \in \binom{V}{2}$, and distance threshold $k \ge 1$, the DCNDP seeks to find a subset of vertices $D \subseteq V$ satisfying the budget $\sum_{v \in D} a_v \le b$ minimizing $\sum_{\{u,v\} \in S} c_{uv}$ with $S := \{\{u,v\} \in \binom{V - D}{2}:\dist_{G[V-D]}(u,v) \le k\}$, $\binom{V - D}{2}$ be the set of all vertex pairs in the set $V-D$, and $\dist_{G[V-D]}(u,v)$ be the hop-based distance between $u$ and $v$ in the induced subgraph $G[V-D]$.
This definition of the DCNDP matches the definition of Class 1 of the DCNDP defined by~\citet{salemi2022} and~\citet{alozie2021distance}.

Both MIP models are specific cases of the \emph{thin} formulation of~\citet{salemi2022} when $k \in \{1,2\}$.
We conduct a polyhedral study and propose aggregated valid inequalities for the 1-DCNDP and 2-DCNDP, respectively.
%

\subsection{The 1-DCNDP MIP formulation}\label{section: 1-hop-connectivity}
The 1-DCNDP formulation minimizes the number of edges whose endpoints are not removed (vaccinated)~\citep{salemi2022}:
\begin{subequations}
\label{1-dcnp}
\begin{align}
\min~&  \sum_{e \in E} x_e\label{dcnp10}\\  
\text{s.t.}~& 1 - y_u - y_v \le x_{uv} && \forall \{u,v\} \in E \label{dcnp11}\\
&\sum_{v \in V} y_v \le b &&  \label{dcnp12}\\
&y_v\in \{0,1\} && \forall v \in V \label{dcnp13}\\
&x\in \mathbb{R}^m_+\label{dcnp14}
\end{align}
\end{subequations}
Objective function~\eqref{dcnp10} minimizes the residual edges (1-hop connections) in the input graph after removing a subset of nodes (i.e., vaccinating a subset of individuals).
Constraints~\eqref{dcnp11} imply that an edge (a 1-hop connection) remains in graph $G$ if none of its endpoints are removed (vaccinated).
Constraint~\eqref{dcnp12} imposes the vaccination budget.
We also note that the integrality of $x$ variables can be relaxed as constraints~\eqref{dcnp11} ensure their integrality. 

The polytope of the 1-DCNDP formulation is defined as follows:
\begin{align}\label{polytope1}
    \mathcal{P}_1 := \left\{x\in \mathbb{R}^m_+,~y \in [0,1]^{n} : (x,y) \text{ satisfies constraints~\eqref{dcnp11}-\eqref{dcnp12}} \right\}
\end{align}
We now conduct a polyhedral study to see which inequalities under what conditions are facet-defining in the MIP of the 1-DCNDP.
As most of these inequalities are facet-defining under mild conditions, one can observe the superiority of their computational performance over that of the 2-DCNDP.

\begin{proposition}
The polytope $\mathcal{P}_1$ is full-dimensional if $b \ge 1$. 
\end{proposition}

\begin{proof}
Full-dimensionality holds by the fact that the following $m+n+1$ points belong to polytope $\mathcal{P}_1$ and are affinely independent: (i) $(\mathbf{1},\mathbf{0})$ (ii) vector $(\mathbf{1} - e_f, e_u)$ for every edge $f=\{u,v\} \in E$, and (iii) vector $(\mathbf{1}, e_v)$ for every vertex $v \in V$. 
In total, we have $n+m+1$ affinely independent points that belong to $\mathcal{P}_1$.
See Table~\ref{table-aff1} for an illustration of the affinely independent points. 
\begin{table}[tbp] 
\centering
\caption{$m+n+1$ affinely independent points that belong to polytope~\eqref{polytope1}.
}
\begin{tabular}{l|c|c|c}
  & 1 point & $m$ points & $n$ points \\ \hline
$\mathbf{x}$ & $\mathbf{1}_{m \times 1}$ & $\mathbf{1}_{m \times m} - \mathbf{I}_{m \times m}$ & $\mathbf{1}_{m \times n}$ \\ \hline
$\mathbf{y}$ & $\mathbf{0}_{m \times 1}$ & $e_1, \dots, e_m$ & $\mathbf{I}_{m \times n}$              
\end{tabular}
\label{table-aff1}
\end{table}

Furthermore, Table~\ref{table-lin1} shows $n+m$ linearly independent points obtained by subtracting point $(\mathbf{1}_{m \times 1}, \mathbf{0}_{m \times 1})$ (the first column of Table~\ref{table-aff1}) from the last $m+n$ points of Table~\ref{table-aff1}.

\begin{table}[tbp] 
\centering
\caption{$m+n$ linearly independent points that are obtained by subtracting point $(\mathbf{1}_{m \times 1}, \mathbf{0}_{m \times 1})$ (the first column of Table~\ref{table-aff1}) from the last $m+n$ points of Table~\ref{table-aff1}.
}
\begin{tabular}{l|c|c}
  & $m$ points & $n$ points \\ \hline
$\mathbf{x}$ & $- \mathbf{I}_{m \times m}$ & $\mathbf{0}_{m \times n}$ \\ \hline
$\mathbf{y}$ & $e_1, \dots, e_m$ & $\mathbf{I}_{m \times n}$              
\end{tabular}
\label{table-lin1}
\end{table}
\end{proof}
Proposition~\ref{trivial_fd} shows that trivial constraints of the MIP formulation of the 1-DCNDP are facet-defining under reasonable conditions.
\begin{proposition}\label{trivial_fd}
The following trivial constraints are facet-defining:
\begin{enumerate}
    \item $x_{uv} \ge 0$ for every edge $\{u,v\} \in E$ if $b \ge 2$
    \item $y_v \ge 0$ for every vertex $v \in V$ if $b \ge 1$
    \item $y_v \le 1$ for every vertex $v \in V$ if $b \ge 2$
\end{enumerate}
\end{proposition}

\begin{proof}
To prove the first statement, choose edge $\{u,v\} \in E$.
Then, Table~\ref{table_aff_trivial_ineq_x} provides $m+n$ affinely independent points that satisfy the inequality $x_{uv} \ge 0$ at equality.
\begin{table}[tbp] 
\centering
\caption{$m+n$ affinely independent points that satisfy inequality $x_{uv} \ge 0$ at equality.
}
\begin{tabular}{l|c|c|c}
  & 3 points & $n-2$ points & $m-1$ points \\ \hline 
$x_{uv}$ & $\mathbf{0}_{1 \times 3}$ & $\mathbf{0}_{1 \times (n-2)}$ & $\mathbf{0}_{1 \times (m-1)}$ \\ \hline
$\mathbf{x}_{(m-1) \times 1}$ & $\mathbf{1}_{(m-1) \times 3}$ & $\mathbf{1}_{(m-1) \times (n-2)}$ & $(\mathbf{1} - \mathbf{I})_{(m-1) \times (m-1)}$ \\ \hline
$y_u$ & $(0,1,1)$ & $\mathbf{1}_{1 \times (n-2)}$ & $\mathbf{1}_{1 \times (m-1)}$ \\ 
$y_v$ & $(1,1,0)$ & $\mathbf{0}_{1 \times (n-2)}$ & $\mathbf{0}_{1 \times (m-1)}$ \\ \hline
$\mathbf{y}_{(n-2) \times 1}$ & $\mathbf{0}_{(n-2) \times 3}$ & $\mathbf{I}_{(n-2) \times (n-2)}$ & $e_1, \dots, e_{m-1}$   
\end{tabular}
\label{table_aff_trivial_ineq_x}
\end{table}

Table~\ref{table_linIndep_trivial_ineq_x} shows $m+n-1$ linearly independent points that are obtained by subtracting the third column of Table~\ref{table_aff_trivial_ineq_x} from other columns. 

\begin{table}[tbp] 
\centering
\caption{$m+n-1$ linearly independent points that are obtained by subtracting the third column of Table~\ref{table_aff_trivial_ineq_x} from other columns.
}
\begin{tabular}{l|c|c|c}
 & 2 points & $n-2$ points & $m-1$ points \\ \hline
 $x_{uv}$ & $\mathbf{0}_{1 \times 2}$ & $\mathbf{0}_{1 \times (n-2)}$ & $\mathbf{0}_{1 \times (m-1)}$ \\ \hline
$\mathbf{x}_{(m-1) \times 1}$ & $\mathbf{0}_{(m-1) \times 2}$ & $\mathbf{0}_{(m-1) \times (n-2)}$ & $- \mathbf{I}_{(m-1) \times (m-1)}$ \\ \hline
$y_u$ & $(-1,0)$ & $\mathbf{0}_{1 \times (n-2)}$ & $\mathbf{0}_{1 \times (m-1)}$ \\ 
$y_v$ & $(1,1)$ & $\mathbf{0}_{1 \times (n-2)}$ & $\mathbf{0}_{1 \times (m-1)}$ \\ \hline
$\mathbf{y}_{(n-2) \times 1}$ & $\mathbf{0}_{(n-2) \times 2}$ & $\mathbf{I}_{(n-2) \times (n-2)}$ & $e_1, \dots, e_{m-1}$   
\end{tabular}
\label{table_linIndep_trivial_ineq_x}
\end{table}

To prove the second statement, choose vertex $v \in V$. 
Then, Table~\ref{table_aff_nonnegativity_ineq_y} provides $m+n$ affinely independent points that satisfy inequality $y_v \ge 0$ at equality when $b \ge 1$.

\begin{table}[tbp] 
\centering
\caption{$m+n$ affinely independent points that satisfy inequality $y_{v} \ge 0$ at equality.
}
\begin{tabular}{l|c|c|c}
 & $m$ points & $n-1$ points & 1 point \\ \hline 
$\mathbf{x}_{m \times 1}$ & $\mathbf{(1-I)}_{m \times m}$ & $\mathbf{1}_{m \times (n-1)}$ & 1 \\ \hline
$y_v$ & $\mathbf{0}_{1 \times m}$ & $\mathbf{0}_{1 \times (n-1)}$ & 0 \\ \hline
$\mathbf{y}_{(n-1) \times 1}$ & $e_1, \dots, e_m$ & $\mathbf{I}_{(n-1) \times (n-1)}$ & $\mathbf{0}_{(n-1) \times 1}$   
\end{tabular}
\label{table_aff_nonnegativity_ineq_y}
\end{table}

Furthermore, Table~\ref{table_linIndep_nonnegativity_ineq_y} shows $m+n-1$ linearly independent points that are obtained by subtracting the last column of Table~\ref{table_aff_nonnegativity_ineq_y} from other columns.

\begin{table}[tbp] 
\centering
\caption{$m+n-1$ linearly independent points that are obtained by subtracting the last column of Table~\ref{table_aff_nonnegativity_ineq_y} from other columns.
}
\begin{tabular}{l|c|c}
 & $m$ points & $n-1$ points \\ \hline
$\mathbf{x}_{m \times 1}$ & $\mathbf{-I}_{m \times m}$ & $\mathbf{0}_{m \times (n-1)}$ \\ \hline
$y_v$ & $\mathbf{0}_{1 \times m}$ & $\mathbf{0}_{1 \times (n-1)}$ \\ \hline
$\mathbf{y}_{(n-1) \times 1}$ & $e_1, \dots, e_m$ & $\mathbf{I}_{(n-1) \times (n-1)}$  
\end{tabular}
\label{table_linIndep_nonnegativity_ineq_y}
\end{table}

To prove the last statement, choose vertex $v \in V$.
Then, Table~\ref{table_aff_atMostOne_ineq_y} provides $m+n$ affinely independent points that satisfy inequality $y_v \le 1$ at equality when $b \ge 2$.

\begin{table}[tbp] 
\centering
\caption{$m+n$ affinely independent points that satisfy inequality $y_{v} \le 1$ at equality.
}
\begin{tabular}{l|c|c|c}
 & $m$ points & $n-1$ points & 1 point \\ \hline
$\mathbf{x}_{m \times 1}$ & $\mathbf{(1-I)}_{m \times m}$ & $\mathbf{1}_{m \times (n-1)}$ & 1 \\ \hline
$y_v$ & $\mathbf{1}_{1 \times m}$ & $\mathbf{1}_{1 \times (n-1)}$ & 1 \\ \hline
$\mathbf{y}_{(n-1) \times 1}$ & $e_1, \dots, e_m$ & $\mathbf{I}_{(n-1) \times (n-1)}$ & $\mathbf{0}_{(n-1) \times 1}$   
\end{tabular}
\label{table_aff_atMostOne_ineq_y}
\end{table}

Furthermore, Table~\ref{table_linIndep_atMostOne_ineq_y} shows $m+n-1$ linearly independent points that are obtained by subtracting the last column of Table~\ref{table_aff_atMostOne_ineq_y} from other columns.

\begin{table}[tbp] 
\centering
\caption{$m+n-1$ linearly independent points that are obtained by subtracting the last column of Table~\ref{table_aff_atMostOne_ineq_y} from other columns.
}
\begin{tabular}{l|c|c}
 & $m$ points & $n-1$ points \\ \hline
$\mathbf{x}_{m \times 1}$ & $\mathbf{-I}_{m \times m}$ & $\mathbf{0}_{m \times (n-1)}$ \\ \hline
$y_v$ & $\mathbf{0}_{1 \times m}$ & $\mathbf{0}_{1 \times (n-1)}$ \\ \hline
$\mathbf{y}_{(n-1) \times 1}$ & $e_1, \dots, e_m$ & $\mathbf{I}_{(n-1) \times (n-1)}$  
\end{tabular}
\label{table_linIndep_atMostOne_ineq_y}
\end{table}

\end{proof}
Proposition~\ref{non_trivial_fd} shows that inequalities~\eqref{dcnp11} are facet-defining under a reasonable condition. 
\begin{proposition}\label{non_trivial_fd}
Inequalities~\eqref{dcnp11} are facet-defining if $b \ge 1$.
\end{proposition}

\begin{proof}
Choose edge $\{u,v\} \in E$.
Then, Table~\ref{table_aff_1b_ineq} provides $m+n$ affinely independent points that satisfy inequality $1 - y_u - y_v \le x_{uv}$ at equality.

\begin{table}[tbp] 
\centering
\caption{$m+n$ affinely independent points satisfying inequality $1 - y_u -y_{v} \le x_{uv}$ at equality.
}
\begin{tabular}{l|c|c|c}
 & 3 points & $m-1$ points  & $n-2$ points \\\hline
$y_u$ & $(0,1,0)$ & $\mathbf{0}_{1 \times (m-1)}$  & $\mathbf{0}_{1 \times (n-2)}$ \\ 
$y_v$ & $(0,0,1)$ & $\mathbf{0}_{1 \times (m-1)}$  & $\mathbf{0}_{1 \times (n-2)}$ \\
$x_{uv}$ & $(1,0,0)$ & $\mathbf{1}_{1 \times (m-1)}$  & $\mathbf{1}_{1 \times (n-2)}$ \\\hline
$\mathbf{x}_R$ & $\mathbf{1}_{(m-1) \times 3}$ & $\mathbf{1 - I}_{(m-1) \times (m-1)}$ & $\mathbf{1}_{(m-1) \times (n-2)}$\\\hline
$\mathbf{y}_R$ & $\mathbf{0}_{(n-2) \times 3}$ & $e_1, e_2, \dots, e_{m-1}$ & $\mathbf{I}_{(n-2) \times (n-2)}$
\end{tabular}
\label{table_aff_1b_ineq}
\end{table}

Furthermore, Table~\eqref{table_linIndep_1b_ineq} shows $m+n-1$ linearly independent points that are obtained by subtracting the first point of Table~\ref{table_aff_1b_ineq} from others.

\begin{table}[tbp] 
\centering
\caption{$m+n-1$ linearly independent points that are obtained by subtracting the first point of Table~\ref{table_aff_1b_ineq} from others.
}
\begin{tabular}{l|c|c|c}
 & 2 points & $m-1$ points & $n-2$ points \\\hline
$y_u$ & $(1,0)$ & $\mathbf{0}_{1 \times (m-1)}$  & $\mathbf{0}_{1 \times (n-2)}$ \\ 
$y_v$ & $(0,1)$ & $\mathbf{0}_{1 \times (m-1)}$  & $\mathbf{0}_{1 \times (n-2)}$ \\
$x_{uv}$ & $(-1,-1)$ & $\mathbf{0}_{1 \times (m-1)}$  & $\mathbf{0}_{1 \times (n-2)}$ \\\hline
$\mathbf{x}_R$ & $\mathbf{0}_{(m-1) \times 2}$ & $\mathbf{-I}_{(m-1) \times (m-1)}$ & $\mathbf{0}_{(m-1) \times (n-2)}$\\\hline
$\mathbf{y}_R$ & $\mathbf{0}_{(n-2) \times 2}$ & $e_1, e_2, \dots, e_{m-1}$ & $\mathbf{I}_{(n-2) \times (n-2)}$
\end{tabular}
\label{table_linIndep_1b_ineq}
\end{table}

\end{proof}

\subsection{The 2-DCNDP MIP formulation}\label{section: 2-hop-connectivity}
In this section, we study the MIP formulation of the 2-DCNDP and propose valid aggregated inequalities to handle high-density instances (i.e., graphs with at least $0.5 \%$ of density in our case study) of the problem.
We first define the edge squared set as follows:
\begin{align*}
    E^2 = \{\{u,v\} \in \binom{V}{2}: \dist_G(u,v) \le 2\}.
\end{align*}
For every vertex $v \in V$, we also define $N_G(v)$ as the neighbor set of vertex $v$.
Then, the 2-DCNDP MIP formulation is provided as follows~\cite{salemi2022}:
\begin{subequations}
\label{2-dcnp}
\begin{align}
\min~&  \sum_{e \in E^2} x_e\label{dcnp00}\\ 
\text{s.t.}~&1 - y_u - y_v - y_i \le x_{uv} &&~\forall\{u,v\} \in E^2 \setminus E,~\forall i \in N_G(u) \cap N_G(v) \label{dcnp01}\\
& 1 - y_u - y_v \le x_{uv} &&~\forall \{u,v\} \in E \label{dcnp02}\\
&\sum_{v \in V} y_v \le b &&  \label{dcnp03}\\
&y_v\in \{0,1\} && \forall v \in V \label{dcnp04}\\
&x_e\in \mathbb{R}_+ && \forall e \in E^2\label{dcnp05}
\end{align}
\end{subequations}
Objective function~\eqref{dcnp00} minimizes the number of connections of length at most two.
Constraints~\eqref{dcnp01}-\eqref{dcnp02} imply that if no connection of length at most two exists between vertices $u$ and $v$, then either (i) $u$ or $v$ or (ii) all of the common neighbors of them are removed (vaccinated). 
Constraint~\eqref{dcnp03} impose the budget constraint.
We note that the integrality of $x$ variables can be relaxed as their integrality is implied by constraints~\eqref{dcnp01}-\eqref{dcnp02}. 

To handle high-density instances and to decrease the number of constraints of type~\eqref{dcnp01}, we propose the following set of inequalities that are aggregated over the common neighbors of $\{u,v\} \in E^2 \setminus E$:
\begin{align}\label{aggregated_ineq}
    &\frac{\sum_{i \in N_G(u) \cap N_G(v)} (1-y_i)}{|N_G(u) \cap N_G(v)|} - y_u - y_v \le x_{uv} &&~\forall\{u,v\} \in E^2 \setminus E
\end{align} 
This aggregation decreases the number of constraints from $O(n^3)$ to $O(n^2)$.
The following lemma shows that constraints~\eqref{aggregated_ineq} are valid.
\begin{lemma}\label{validity}
    Aggregated inequalities~\eqref{aggregated_ineq} are valid.
\end{lemma}
\begin{proof}
    Let $(\hat{x}, \hat{y})$ be a feasible solution of the 2-DCNDP. For every vertex pair $\{u,v\} \in E^2 \setminus E$, we consider the following cases:
    \begin{itemize}
        \item if $\hat{x}_{uv} = 1$, then inequality~\eqref{aggregated_ineq} is satisfied trivially.
        \item if $\hat{x}_{uv} = 0$, then we have either of the following cases:
        \begin{itemize}
            \item if either $\hat{y}_u = 1$ or $\hat{y}_v = 1$, then inequality~\eqref{aggregated_ineq} is satisfied trivially.
            \item if $\hat{y}_u = 0$ and $\hat{y}_v = 0$, then $\hat{y}_i = 1$ for every common neighbor $i \in N_G(u) \cap N_G(v)$ by the feasibility of $(\hat{x}, \hat{y})$.
        \end{itemize}
    \end{itemize}
\end{proof}

Now we prove that for any \emph{integer feasible solution}, the aggregated inequalities~\eqref{aggregated_ineq} imply inequalities~\eqref{dcnp01}.

\begin{lemma}\label{implied_ineq}
    Let $(\hat{x}, \hat{y})$ be an integer feasible solution that satisfies the aggregated inequalities~\eqref{aggregated_ineq}.
    Then, $(\hat{x}, \hat{y})$ satisfies inequalities~\eqref{dcnp01}.
\end{lemma}

\begin{proof}
    For every vertex pair $\{u,v\} \in E^2 \setminus E$, we consider the following cases:
    \begin{itemize}
        \item if $\hat{x}_{uv} = 1$, then inequality~\eqref{dcnp01} is satisfied trivially.
        \item if $\hat{x}_{uv} = 0$, then we have either of the following cases:
        \begin{itemize}
            \item if either $\hat{y}_u = 1$ or $\hat{y}_v = 1$, then inequality~\eqref{dcnp01} is satisfied trivially.
            \item if $\hat{y}_u = 0$ and $\hat{y}_v = 0$, then $\hat{y}_i = 1$ for every common neighbor $i \in N_G(u) \cap N_G(v)$ by the feasibility of $(\hat{x}, \hat{y})$.
        \end{itemize}
    \end{itemize} 
\end{proof}

The following theorem shows that one can replace Constraints~\eqref{dcnp01} with constraints~\eqref{aggregated_ineq}.

\begin{theorem}
    Constraints~\eqref{dcnp01} can be replaced with constraints~\eqref{aggregated_ineq}.
\end{theorem}
\begin{proof}
    The proof follows by Lemmas~\ref{validity} and~\ref{implied_ineq}.
\end{proof}
Section~\ref{subsection: mip formulations} discusses how aggregated inequalities improve the computational performance of the 2-DCNDP MIP for graphs with densities of at least $0.5 \%$.

\section{Computational Experiments: A case study for Newfoundland and Labrador}\label{section: experiment} 
We conduct our computational experiments on a simulated contact network of individuals in Newfoundland and Labrador (NL), generated by the Medical Operations Research Lab's Pandemic Outbreak Planner (morPOP) using census, workplace, school, travel, and other various population data for NL~\cite{aleman2020morpop}. 
morPOP is a state-of-the-art granular agent-based simulation (ABS) modeling tool used to generate realistic contact networks between individuals. 
It constructs human interaction networks by considering population density, community structures, and individual behavior reflected by census data. 
We use the 2016 census of NL~\cite{statcan_NL_2016} to generate population snapshots. 
Each snapshot consists of a network of 507,555 individuals (the population of NL at the time) represented in the edge list format. Detailed computational results for the MIP formulations of the 1-DCNDP and 2-DCNDP, along with codes and data, are  available at~\href{https://github.com/faraz2023/optimal-design-of-vaccination-policies}{\url{https://github.com/faraz2023/optimal-design-of-vaccination-policies}}.

To handle the high-density instances of the problem, we first decompose the input graph into multiple partitions based on the regional health authorities (RHAs).
This assumption is reasonable as health regions are units of health governance for administration and delivery of public health in the Canadian healthcare model~\cite{wiki2023}.
NL has four RHAs: Eastern Health, Central Health, Western Health, and Labrador-Grenfell Health.
When the size of each health authority is not computationally tractable by the MIPs of 1-DCNDP and 2-DCNDP, we employ the \textit{NetworkX-METIS} implementation of multilevel multiway partitioning algorithm that minimizes the number of cross-partition edges~\cite{NetworkX-METIS2022}. 
Finally, we employ a simplicial fixing procedure (see \citet{salemi2022}, Proposition 3) as soon as the partition can be handled by the 1-DCNDP and 2-DCNDP MIPs.
For each ``handlable'' partition, we set the budget $b$ to $20\%$ of the number of nodes in the partition. 

We run our experiments on Dell XPS 8940 desktop with  Windows 11 PC Pro, Core i9-111900K CPU (8 cores, 16 threads, 3.50Ghz 5.10GHz Turbo, 16MB, 95W), and 64GB DDR4 RAM.
Our codes are written in Python 3.10.9.
We employ Gurobi 10.0.1 to run the MIP formulations of 1-DCNDP and 2-DCNDP instances.
The Gurobi time limit is set to one hour for all instances of the 1-DCNDP and 2-DCNDP MIP formulations.
Due to the space limitation, we present only our computational results with aggregated inequalities of the MIP formulation for the 2-DCNDP in Section~\ref{subsection: mip formulations}.   

\subsection{Experiments with aggregated inequalities for the MIP of the 2-DCNDP}\label{subsection: mip formulations}
As mentioned in Section~\ref{section: 2-hop-connectivity}, we proposed aggregated inequalities to make the MIP size of the 2-DCNDP smaller in terms of the number of constraints.
In terms of optimality gap, Table~\ref{aggregated_results} shows the superiority of the 2-DCNDP MIP with aggregated inequalities over disaggregated inequalities for 10 out of 15 instances.
Interestingly, the 2-DCNDP MIP with aggregated inequalities solves partition 18 of Eastern Health to optimality while the disaggregated variant does not solve even the LP relaxation of the instance in one hour.  
For three instances (i.e., partition 12 of Labrador-Grenfell, partition 42 of Eastern Health, and partition 121 of Eastern Health), the optimality gaps are the same for both approaches.
However, the aggregated approach provides better objective values for two of these three instances.
While we observe that the disaggregated approach outperforms for two instances (i.e., partition 58 of Eastern Health and partition 29 of Western Health) in terms of optimality gap, the aggregated approach provides better objective values for both instances.
In terms of objective value, the aggregated variant of the 2-DCNDP MIP works at least as well as the disaggregated variant for all instances except partition 16 of Eastern Health.

\begin{table}[tbp]
\centering
\caption{Computational results for the disaggregated inequalities~\eqref{dcnp01} and aggregated inequalities~\eqref{aggregated_ineq} of the MIP formulation of the 2-DCNDP. Part: partition ID; $d (\%)$: density; 2-hop dec ($\%$): 2-hop decrease; LPNS: LP relaxation of the 2-DCNDP MIP is not solved in an hour. Bold, underlined values are the better results of aggregated v.\ disaggregated inequalities.}
\begin{tabularx}{1\textwidth}{l|r|r|r|r|rr|rr}
\hline
& 
& 
& 
& 
& \multicolumn{2}{c|}{2-hop dec (\%)} & \multicolumn{2}{c}{Opt gap (\%)} \\ \cline{6-9}
RHA & Part. & $n$ & $m$ & $d~(\%)$ & Agg. & Disagg. & Agg. & Disagg. \\ \hline
East  & 16 & 2,456 & 15,137 & 0.50 & 91.48 & \underline{\textbf{91.50}} & \underline{\textbf{5.96}} & 6.27\\
East & 13 & 2,550 & 16,476 & 0.51 & \underline{\textbf{89.29}} & 67.29 & \underline{\textbf{47.00}} & 100.00\\
La-Gr & 12 & 2,420 & 16,501 & 0.56 & \underline{\textbf{63.21}} & 60.44 & 100.00 & 100.00\\
East & 58 & 2,476 & 17,116 & 0.56 & \underline{\textbf{97.61}} & 97.60 & 0.68 & \underline{\textbf{0.55}}\\
West & 9 & 2,502 & 18,799 & 0.60 & \underline{\textbf{95.84}} & 95.78 & \underline{\textbf{9.21}} & 11.97\\
La-Gr & 3 & 2,420 & 19,387 & 0.66 & \underline{\textbf{91.61}} & 88.29 & \underline{\textbf{45.19}} & 63.40\\
East & 42 & 2,432 & 19,478 & 0.66 & 96.54 & 96.54 & 0.00 & 0.00\\
Cent & 35 & 2,508 & 22,103 & 0.70 & \underline{\textbf{95.72}} & 68.28 & \underline{\textbf{36.58}} & 100.00\\
East & 96 & 2,539 & 22,497 & 0.70 & \underline{\textbf{93.35}} & 66.66 & \underline{\textbf{59.57}} & 100.00\\
Cent & 8 & 2,513 & 22,460 & 0.71 & \underline{\textbf{92.86}} & 63.32 & \underline{\textbf{49.63}} & 100.00\\
West & 1 & 2,350 & 21,749 & 0.79 & \underline{\textbf{93.47}} & 92.68 & \underline{\textbf{9.90}} & 20.91\\
East & 116 & 2,567 & 33,030 & 1.00 & 92.44 & 92.44 & \underline{\textbf{72.05}} & 72.34\\
West & 29 & 2,504 & 31,833 & 1.02 & \underline{\textbf{80.71}} & 80.35 & 63.07 & \underline{\textbf{62.04}}\\
East & 121 & 2,547 & 37,497 & 1.16 & \underline{\textbf{46.44}} & 46.05 & 100.00 & 100.00\\
East & 18 & 2,457 & 74,868 & 2.48 & \underline{\textbf{99.31}} & 98.88 & \underline{\textbf{0.00}} & LPNS \\
\hline
\end{tabularx}
\label{aggregated_results}
\end{table}

\subsection{Strategic vaccine allocation policies via decision trees}\label{subsection: ml}

We implement a pipeline based on decision trees to explain (i.e., predict) the DCNDP solutions via the simulated individuals' attributes and demographic information.
This process facilitates the construction of pragmatic vaccination strategies, potentially valuable for public health policymakers. 
The DCNDP solutions are selected solely based on the topology of the contact network, which itself is a statistical representation of the population, and are not implementable in real-world vaccine strategies as these topological properties resulting in DCNDP selection are not easily describable to policymakers and individuals in a real population.
To transform DCNDP solutions to actionable policies, we train decision trees to predict the DCNDP label (vaccinate or not vaccinate) of the individuals via the rich and granular personal attributes simulated by morPOP for each node.

We formulate the prediction task of the DCNDP solutions as a transductive (i.e., testing samples are available during training) supervised learning scenario and use the \texttt{scikit-learn} library~\cite{scikit-learn} in Python to implement the decision trees. In our experience, trees with a maximum depth of five based on entropy criterion yield the best performance. In total, we train 10 models: a whole-population tree and four per-RHA trees for each of the 1-DCNDP and 2-DCNDP solutions.

We study each decision tree manually and derive rollout phases with clear boundaries. A succession of rollout phases constitutes a vaccination policy. Remarkably, all 10 decision trees yield similar interpretations for rollout phases\footnote{Visualizations of all decision trees are available on the GitHub repository.}.
While the DCNDP-inspired rollout is a raw and analytical implication of the trained trees that is obtained by solely minimizing a network connectivity metric, policymakers pursue multi-objective vaccination policies that strike a balance between disease transmission and the risk to the frontline workers or the high-risk population. 
Motivated by these real-world considerations, we also provide an alternative DCNDP rollout (which we call DCNDP-realistic rollout) in which we manually impose early vaccination for healthcare workers, urgent-care patients, and individuals above age 80 (see Figure~\ref{fig:dt_viz}). 
We evaluate our DCNDP-inspired vaccine rollouts against the real-world 3-phase vaccination plan implemented by the Government of NL during the COVID-19 pandemic (Figure~\ref{fig:3phase_vis}).

\begin{figure}[tbp]
    \centering
    \begin{subfigure}{.52\textwidth}
        \centering
        \includegraphics[width=1\linewidth]{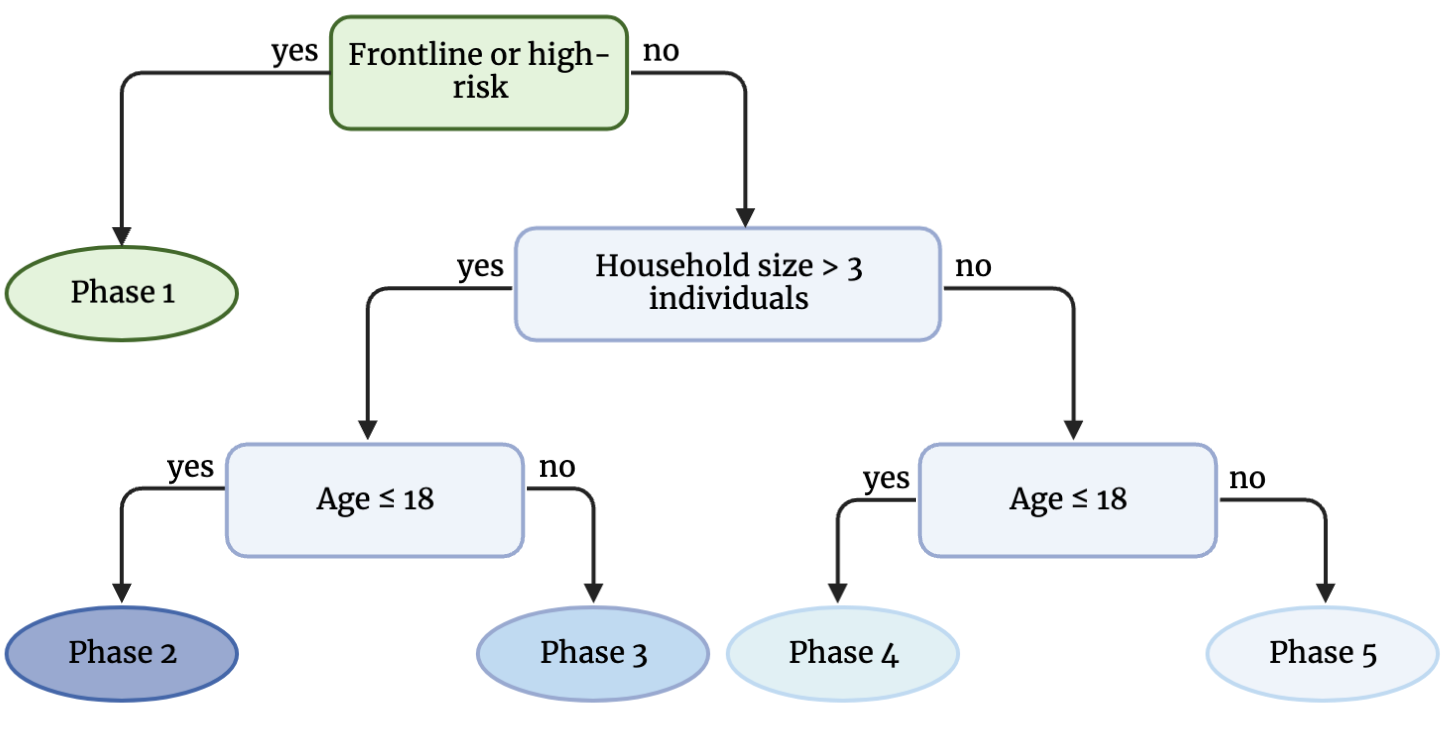}
        \caption{DCNDP-realistic rollout}
        \label{fig:dt_viz}
    \end{subfigure}%
    \hfill
    \begin{subfigure}{.47\textwidth}
        \centering
        \includegraphics[width=1\linewidth]{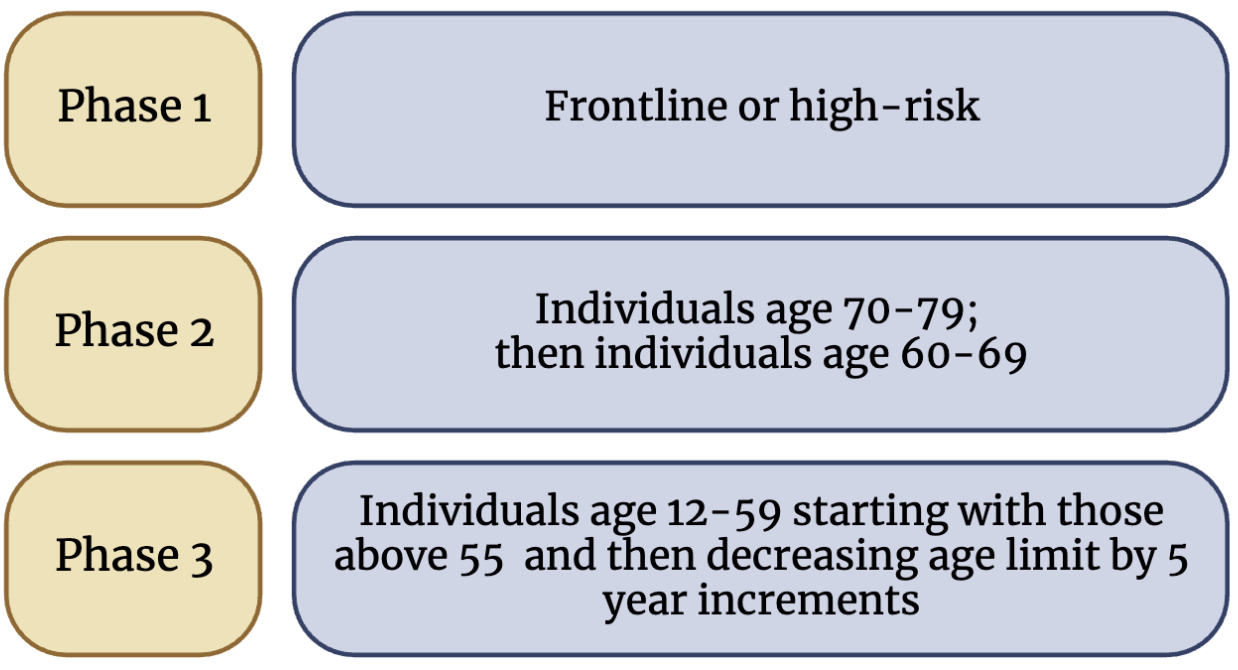}
        \caption{Real-world 3-phase vaccination rollout ``baseline'' implemented by NL\protect\footnotemark}
        \label{fig:3phase_vis}
    \end{subfigure}
    \caption{Vaccination prioritization strategies. We note that phases two and three of the real-world strategies have two and 10 inner stages, respectively.
    }
    \label{fig:vaccine_policies}
\end{figure}

\footnotetext{\url{https://www.gov.nl.ca/covid-19/vaccine/prioritygroups/}}

We evaluate each vaccination rollout policy as a 100-day plan where the daily vaccination budget is set to $1\%$ of the population. Each day we randomly vaccinate individuals from the current rollout phase by removing them from the network (for simplicity, we assume a vaccine efficacy of $100\%$). Figures~\ref{fig:vaccine_1hop}-\ref{fig:vaccine_R0} provide a comparison between the two DCNDP rollouts and real-world baseline with respect to three metrics: 1-hop connections, 2-hop connections, and $R_0$. We calculate $R_0$ as follows.
\begin{equation}
R_0=T\left(\frac{\sum_{d=1}^{\infty} p_d d^2}{\sum_{d=1}^{\infty} p_d d}-1\right)=T\left(\frac{\left\langle d^2\right\rangle}{\langle d\rangle}-1\right)
\label{eq:r0}
\end{equation}
where $p_d$ is the observed probability of degree $d$, $T$ is disease transmissibility, and $\left\langle d\right\rangle$ and $\left\langle d^2\right\rangle$ are the mean degree and the mean squared degree of the network, respectively~\cite{newman2002spread}. We note that the disease transmissibility $T$ is calculated independently of network structure, and therefore, the $R_0$ calculation relies on secondary methods~\cite{li2011failure}. In this work, we conduct our analysis independent of any particular disease (i.e., $T=1$) to demonstrate an upper bound for the worst-case disease scenario~\cite{ventresca2013evaluation}.

Table~\ref{tab:vaccine_results} demonstrates that our proposed DCNDP-realistic vaccine rollout can have a better performance in comparison with the real-world baseline rollout in terms of both 1-DCNDP and 2-DCNDP metrics as well as $R_0$.
While Figures~\ref{fig:vaccine_1hop} and~\ref{fig:vaccine_2hop} show the superiority of both DCNDP approaches over the real-world baseline in terms of 1-hop and 2-hop connections, we observe no significant difference between the curves of DCNDP and DCNDP-realistic rollouts.

Interestingly, Figure~\ref{fig:vaccine_R0} shows that $R_0$ is not monotonically decreasing, which is counterintuitive since more vaccinations (i.e., node and link removals) should not increase a network transmissibility metric.
This numerical instability is due to the dependence of $R_0$ calculations on the degree distribution of a given network and that the ratio between $\left\langle d\right\rangle$ and $\left\langle d^2\right\rangle$ in Equation~\eqref{eq:r0} can behave unpredictably when nodes are changed to have degree zero.
It is also worth noting that $R_0$ can be misleading for studying vaccine policy implementation since the metric is designed only as a threshold to predict whether a pandemic will die down or become an endemic~\cite{ventresca2013evaluation}.

\begin{table}[tbp]
\centering
\caption{Analysis of the DCNDP-realistic vaccine rollout and the real-world baseline rollout. Area (\%): percentage of area shrinkage between the two curves.}
\resizebox{\textwidth}{!}{%
\begin{tabular}{l|c|r|r|r|r}
\hline
& & & \multicolumn{1}{c|}{Days with} & \multicolumn{2}{c}{Daily improv. (\%)} \\\cline{5-6}
RHA & Metric & Area (\%) & improv. (\%) & Avg.  & Std. \\\hline
\multirow{3}{*}{Overall NL} 
& 1-hop & 27.75 & 92.23 & 43.62 & 27.23 \\
& 2-hop & 23.99 & 92.23 & 42.73 & 27.99 \\
& $R_0$ & 28.32 & 91.26 & 32.77 & 30.01 \\
\hline
\multirow{3}{*}{East} 
& 1-hop & 28.33 & 93.14 & 44.01 & 27.77 \\
& 2-hop & 22.22 & 93.14 & 42.13 & 28.91 \\
& $R_0$ & 29.59 & 92.16 & 33.78 & 31.23 \\
\hline
\multirow{3}{*}{Cent} 
& 1-hop  & 24.93 & 91.18 & 40.93 & 27.83 \\
& 2-hop  & 20.08 & 91.18 & 39.25 & 29.06 \\
& $R_0$ & 23.37 & 82.35 & 27.50 & 31.97 \\
\hline
\multirow{3}{*}{West} 
& 1-hop  & 25.27 & 91.18 & 41.84 & 28.12 \\
& 2-hop  & 21.61 & 91.18 & 41.27 & 28.93 \\
& $R_0$ & 24.88 & 91.18 & 29.68 & 31.30 \\
\hline
\multirow{3}{*}{La-Gr} 
& 1-hop  & 15.51 & 94.12 & 28.97 & 24.95 \\
& 2-hop  & 5.01 & 65.69 & 14.95 & 29.87 \\
& $R_0$ & 0.30 & 49.02 & 5.06 & 33.83 \\\hline
\end{tabular}
}
\label{tab:vaccine_results}
\end{table}

\begin{figure}[tbp]
    \centering
    \begin{subfigure}{.51\textwidth}
        \centering
        \includegraphics[trim={3cm 3cm 3cm 3cm},width=1\linewidth]{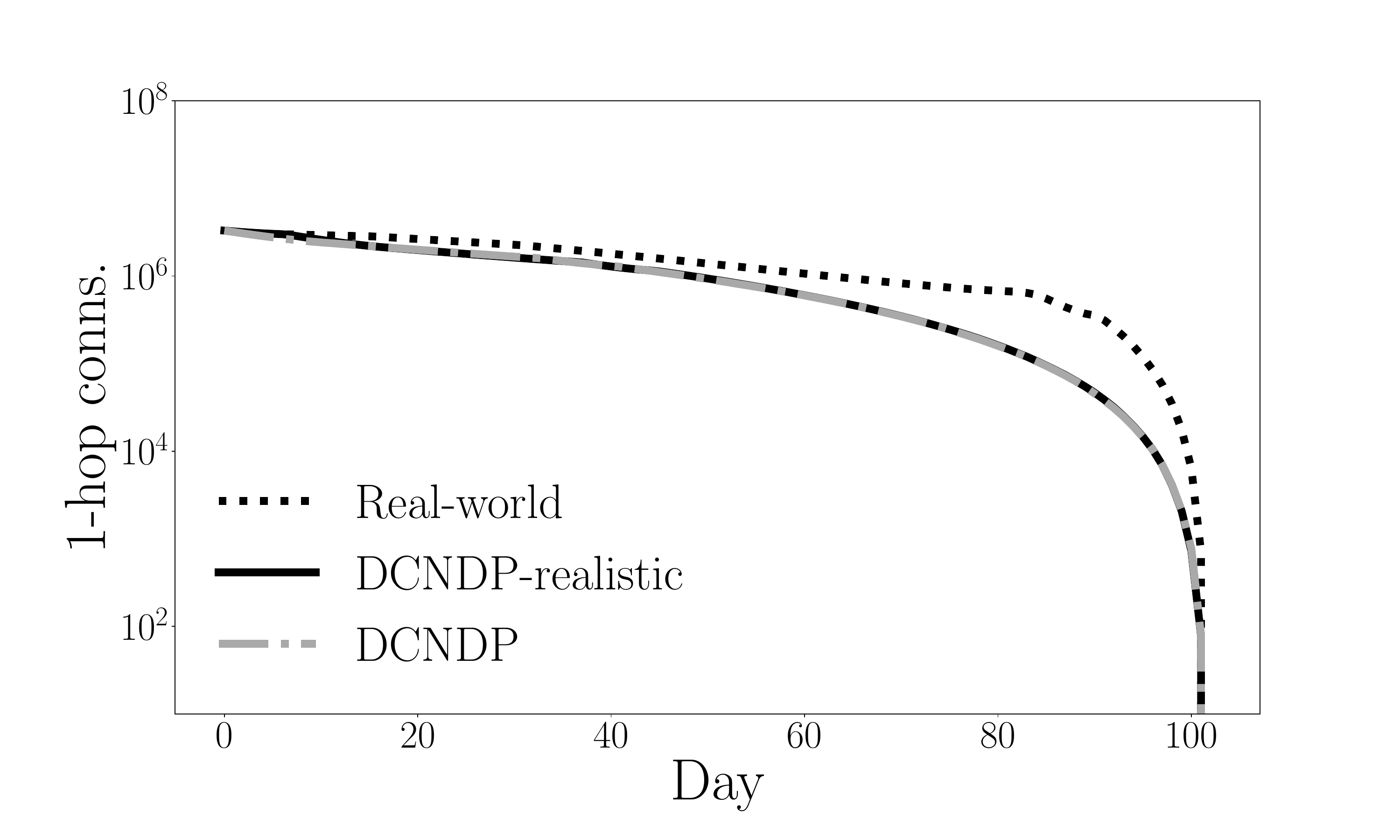}
        \vspace{2pt}
        \caption{NL total}
        \label{fig:HA_1hop}
    \end{subfigure}%
    \hfill
    \begin{subfigure}{.47\textwidth}
        \centering
        \includegraphics[trim={3cm 3cm 3cm 3cm},width=1\linewidth]{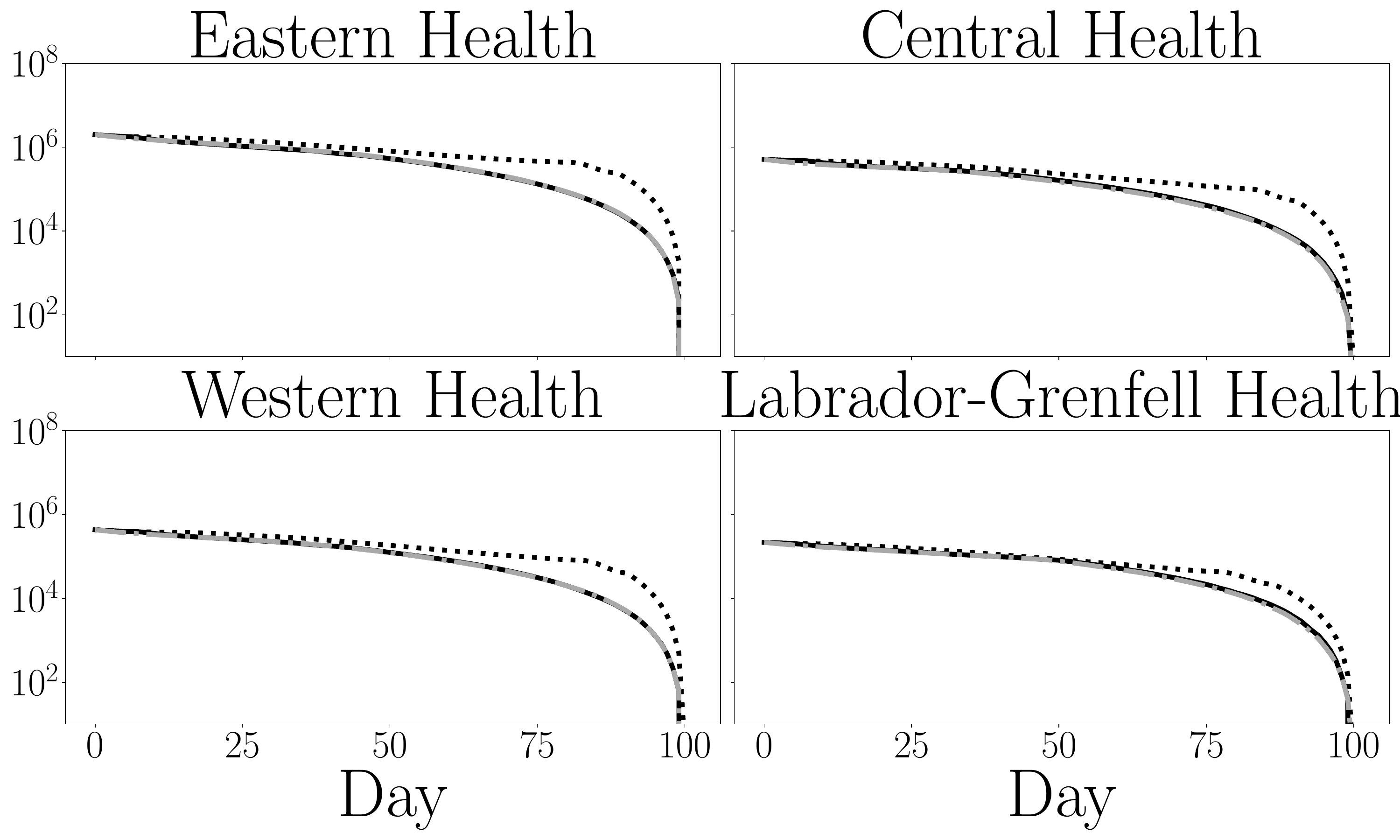}
        \vspace{2pt}
        \caption{Individual RHAs}
        \label{fig:overall_1hop}
    \end{subfigure}
    \caption{1-hop connection performance based on vaccine rollout strategies}
    \label{fig:vaccine_1hop}
\end{figure}

\begin{figure}[tbp]
    \centering
    \begin{subfigure}{.51\textwidth}
        \centering
        \includegraphics[trim={3cm 3cm 3cm 3cm},width=1\linewidth]{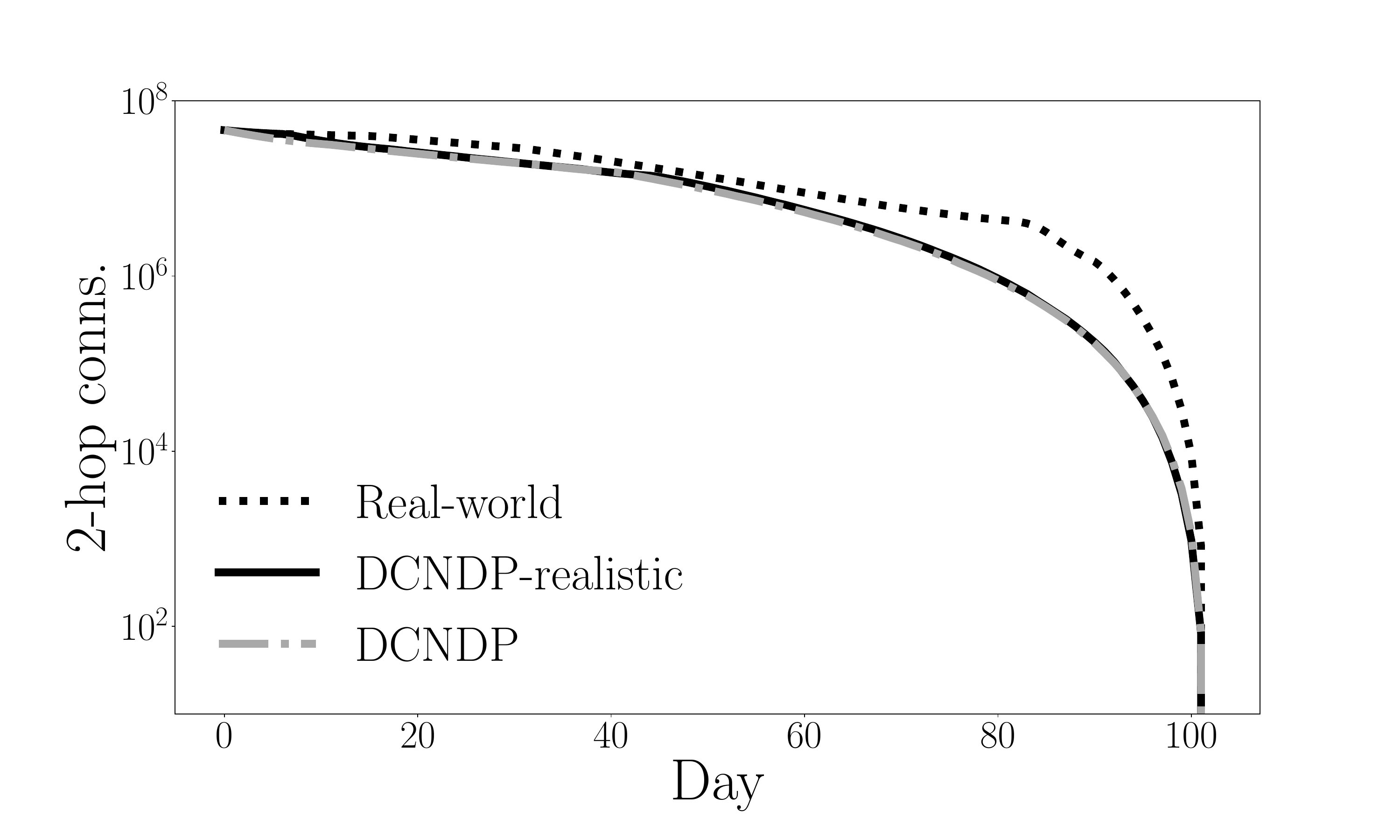}
        \vspace{2pt}
        \caption{NL total}
        \label{fig:HA_2hop}
    \end{subfigure}%
    \hfill
    \begin{subfigure}{.47\textwidth}
        \centering
        \includegraphics[trim={3cm 3cm 3cm 3cm},width=1\linewidth]{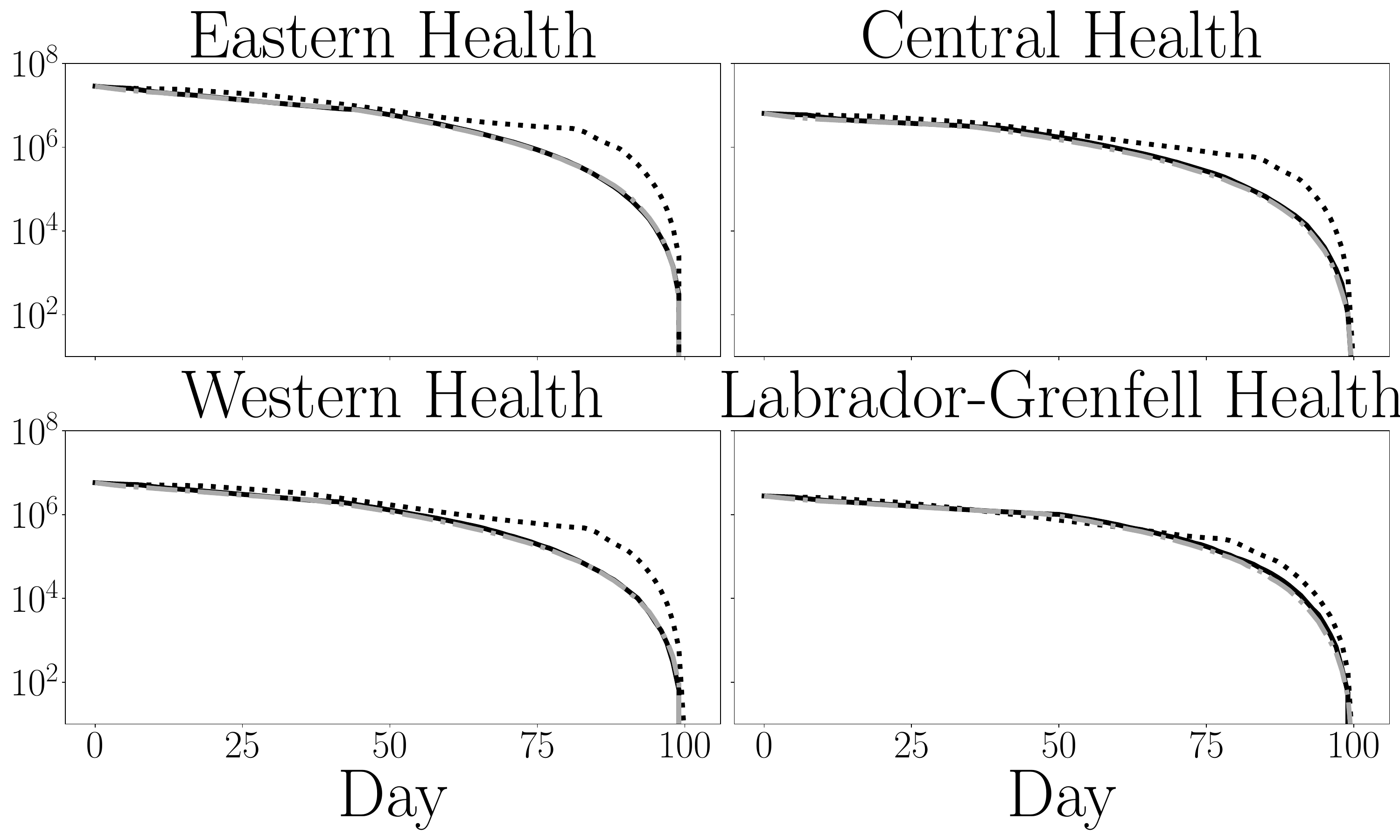}
        \vspace{2pt}
        \caption{Individual RHAs}
        \label{fig:overall_2hop}
    \end{subfigure}
    \caption{2-hop connection performance based on vaccine rollout strategies}
    \label{fig:vaccine_2hop}
\end{figure}

\begin{figure}[tbp]
    \centering
    \begin{subfigure}{.51\textwidth}
        \centering
        \includegraphics[trim={3cm 3cm 3cm 3cm},width=1\linewidth]{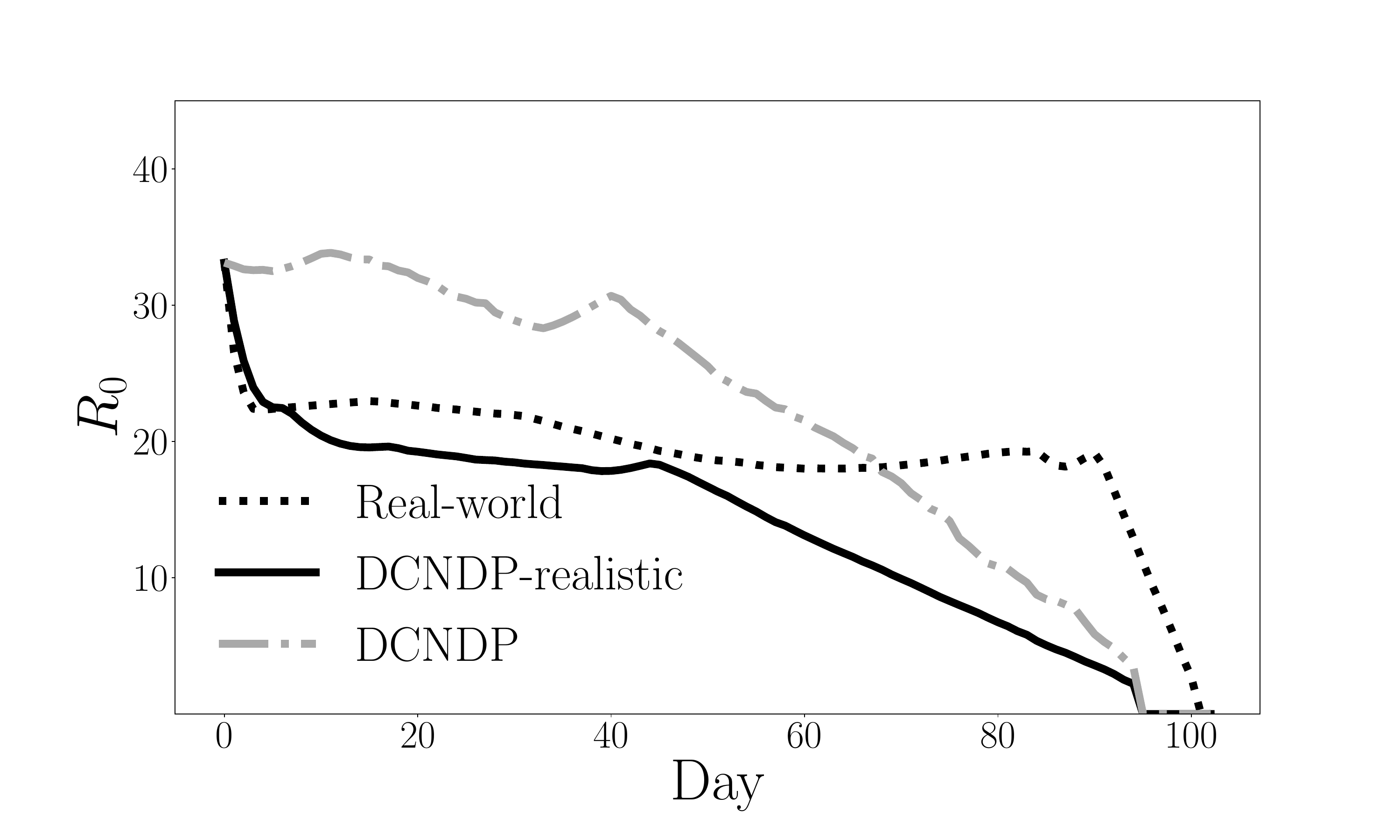}
        \vspace{2pt}
        \caption{NL total}
        \label{fig:HA_R0}
    \end{subfigure}%
    \hfill
    \begin{subfigure}{.47\textwidth}
        \centering
        \includegraphics[trim={3cm 3cm 3cm 3cm},width=1\linewidth]{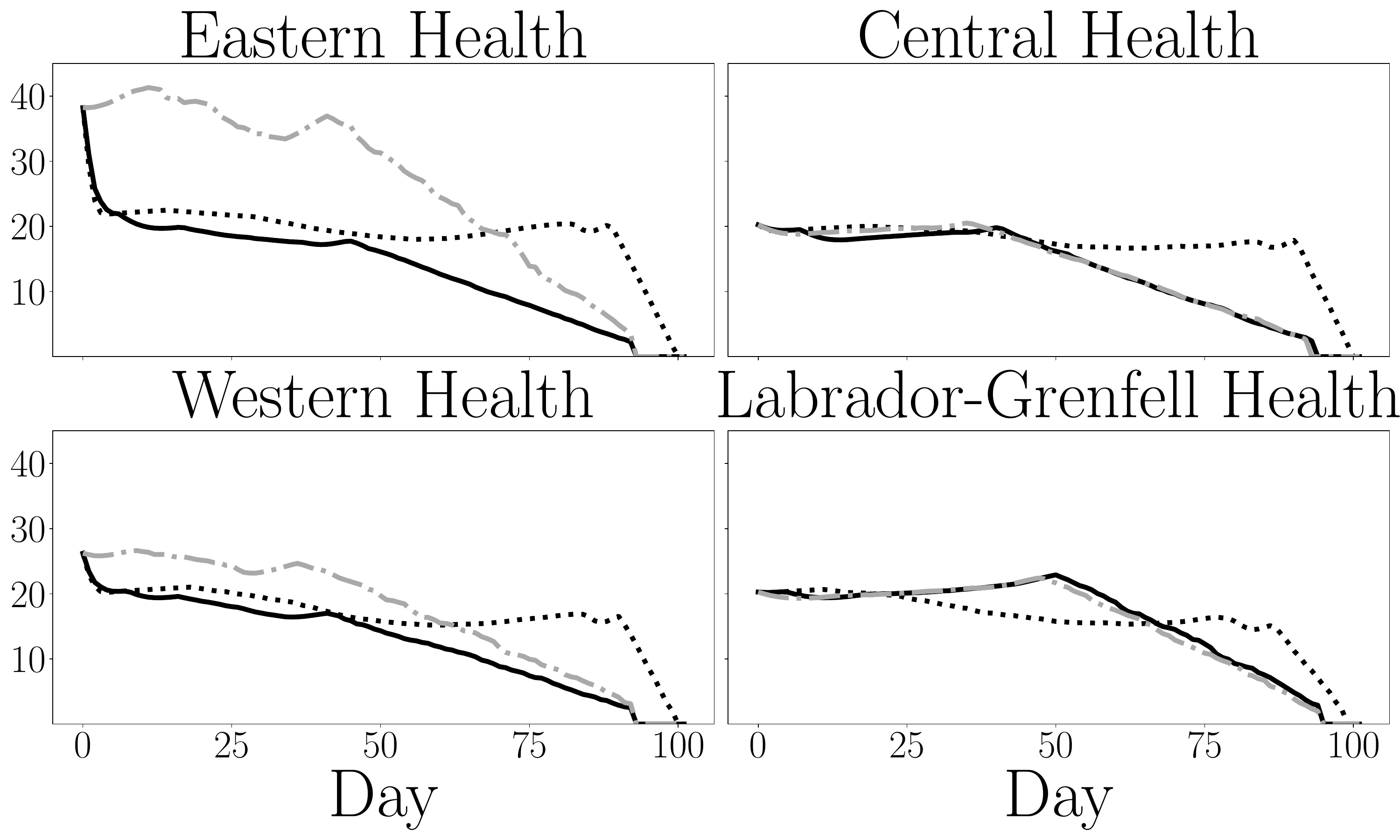}
        \vspace{2pt}
        \caption{Individual RHAs}
        \label{fig:overall_R0}
    \end{subfigure}
    \caption{$R_0$ performance based on vaccine rollout strategies
    }
    \label{fig:vaccine_R0}
\end{figure}

\section{Conclusion and Future Work}
This paper explores the boundary of MIP formulations of DCNDP to minimize the number of connections of length at most one or two in a vaccination context.
When minimizing the number of connections of length at most one is desired, we provided a polyhedral study for its corresponding MIP formulation and proved when the trivial and non-trivial constraints are facet-defining.
For the DCNDP with the objective of minimizing the number of connections of length at most two, we reduced the number of a set of constraints in an order of magnitude and proposed a new set of valid aggregated inequalities.
Our computational results in a vaccination context show the superiority of the aggregated inequalities over the disaggregated ones when the density of the input graph is at least $0.5\%$.
We finally employed decision trees by leveraging rich and granular data afforded by morPOP to construct potential strategic vaccine rollout policies for Newfoundland and Labrador, demonstrating that DCNDP-driven policies are more effective than common real-world strategies for mitigating disease spread. However, it is important to note that real-world policies also considered limiting mortality and severity, not just infections, though our DCNDP-realistic policy attempted to capture that concern by prioritizing healthcare workers and vulnerable individuals in the first priority group. We additionally found that $R_0$ is an unreliable metric of disease spread potential in a critical node context, due to its unstable behavior in the presence of zero-degree nodes.

As future work, one can explore other variants of the DCNDP in a vaccination context and apply MIP tools and techniques (e.g., warm-starting MIPs with quality heuristic solutions and reducing the number of decision variables and constraints to decrease the size of the MIP formulations) to them. 
Additionally, the evaluation of decision trees for deriving vaccination policies has been done manually in this work and can be automated in future work.

\section*{Acknowledgments}

We thank Dr.\ Proton Rahman (Eastern Health, NL) and Dr.\ Randy Giffen (IBM Canada) for their help in building the morPOP agent-based simulation model for COVID-NL and for their expertise in COVID mitigation strategies. We also thank Newfoundland \& Labrador Public Health for their guidance and data contributions to morPOP.

\bibliographystyle{elsarticle-num-names} 
\bibliography{main}

\end{document}